\documentclass[11 pt]{article}

\usepackage[latin1]{inputenc}
\usepackage{amsmath,amsfonts,hyperref,amsthm, graphicx}
\usepackage[textsize=footnotesize,color=green!40]{todonotes}
\usepackage{fullpage}
\usepackage{authblk}
\usepackage{changes}
\usepackage{bbm}
\usepackage{wrapfig}
\usepackage{caption}
\usepackage{amssymb}
\usepackage{cleveref}
\usepackage[shortlabels]{enumitem}

\newtheorem{theorem}{Theorem}[section]
\newtheorem{lemma}[theorem]{Lemma}
\newtheorem{conditions}[theorem]{Condition}

\newtheorem{conjecture}[theorem]{Conjecture}

\theoremstyle{definition}
\newtheorem{definition}[theorem]{Definition}

\newtheorem{claim}[theorem]{Claim}
\newtheorem*{claim*}{Claim}

\newcommand{\cB}{\mathcal{B}}
\newcommand{\cE}{\mathcal{E}}
\newcommand{\cF}{\mathcal{F}}
\newcommand{\cG}{\mathcal{G}}
\newcommand{\cH}{\mathcal{H}}

\newcommand{\bE}{\mathbb{E}}

\newcommand{\cL}{\mathcal{L}}

\newcommand{\cC}{\mathcal{C}}
\newcommand{\cQ}{\mathcal{Q}}

\newcommand{\bd}{\mathbf{d}}

\newcommand{\sm}{\setminus}

    \makeatletter
    \let\@fnsymbol\@arabic
    \makeatother

\newcommand{\eps}{\varepsilon}

\newcommand{\bP}{\mathbb{P}}

\newcommand{\cM}{\mathcal{M}}
\newcommand{\bN}{\mathbb{N}}

\newtheoremstyle{case}{}{}{}{}{}{:}{ }{}
\theoremstyle{case}

\newcommand{\Bin}{\text{Bin}}
\newcommand{\aut}{\text{aut}}
\newcommand{\bbl}{\bigg(}
\newcommand{\bbr}{\bigg)}
\newcommand{\dtv}{d_{\text{TV}}}
\newcommand{\excess}{\text{excess}}

\author{Matthew Coulson\thanks{Department of Combinatorics \& Optimization, University of Waterloo, Canada. Email: matthew.coulson@uwaterloo.ca.}}

\begin{document}
\title{The strong component structure of the barely subcritical directed configuration model}

\maketitle
\begin{abstract}
We study the behaviour of the largest components of the directed configuration model in the barely subcritical regime.
We show that with high probability all strongly connected components in this regime are either cycles or isolated vertices and give an asymptotic distribution of the size of the $k$th largest cycle.
This gives a configuration model analogue of a result of \L{}uczak and Seierstad for the binomial random digraph.
\end{abstract}

\section{Introduction}\label{sec:intro}
\subsection{The directed configuration model}
Let $\bd_n= (\bd_n^-,\bd_n^+)=((d_1^-,d_1^+),\ldots,(d_n^-,d_n^+))$ be a directed degree sequence on $n$ vertices and let $m_n = \sum d_i^- = \sum d_i^+$.
The directed configuration model on $\bd_n$, introduced by Cooper and Frieze~\cite{CF}, is the random directed multigraph on $[n]$ obtained by associating with vertex $i$ $d_i^-$ in-stubs and $d_i^+$ in-stubs, and then choosing a perfect matching of in- and out- stubs uniformly at random.
We will denote this model by $G(\bd_n)$.
This is a directed generalisation of the configuration model of Bollob\'as~\cite{B} which since its introduction has become one of the most widely used random graph models.

A strongly connected component in a digraph is a maximal sub-digraph such that there exists a directed path between each ordered pair of vertices.
In this paper we will consider the size of the largest strongly connected component in the barely subcritical regime.
Note that these components are much smaller than the connected components which are found in the undirected case.
This is due to the fact that the minimal strongly connected components are directed cycles as opposed to the trees in the directed case and so we require one additional edge on the same vertex set.
We shall call any strongly connected digraph which is not a cycle complex.
Note that any such digraph has strictly more edges than vertices.

Let $n_{k,\ell}$ be the number of copies of $(k,\ell)$ in $\bd_n$ and let $\Delta_n = \max (d_i^-, d_i^+)$ be the maximum degree of $\bd_n$.
Also, define the following parameters of the degree distribution.
These parameters govern the behaviour of the size of the largest strongly connected component.
\begin{definition}
 \begin{align*}
  Q_n  := \frac{1}{m_n} \sum_{i=1}^n d_i^-d_i^+ - 1 && R_n^-  := \frac{1}{m_n} \sum_{i=1}^n d_i^-d_i^+(d_i^- -1) && R_n^+  := \frac{1}{m_n} \sum_{i=1}^n d_i^-d_i^+(d_i^+ -1) 
 \end{align*}
We shall assume the following conditions which ensure that the degree sequence is suitably ``well behaved''.
\begin{conditions}\label{cond:conditions}
For each $n$, $(d_i^-,d_i^+)_{i=1}^n = ((d_i^-,d_i^+)^{(n)})_{i=1}^n$ is a sequence of ordered pairs of non-negative integers such that $\sum_{i=1}^n d_i^- = \sum_{i=1}^n d_i^+$.
Furthermore, $(p_{i,j})_{i,j=1}^\infty$ is a probability distribution such that for some $\eps, \zeta >0$,
 \begin{enumerate}[i)]
  \item $n_{i,j}/n \to p_{i,j}$ as $n \to \infty$ for each $i,j \geq 0$,
  \item $m_n/n = \mu^{(n)} \to \mu = \sum_{i,j=1}^\infty i p_{i,j} = \sum_{i,j=1}^\infty j p_{i,j}$,
  \item $n_{0,0} =0$,
  \item $\sum_{i=1}^\infty n_{0,i} + n_{i,0} \leq (1-\eps)n$,
  \item $n_{1,1} \leq (1-\eps)n$,
  \item $\Delta_n \leq n^{1/6} \log^{-1/4}(n)$,
  \item $R_n^-, R_n^+ \geq \zeta$.
 \end{enumerate}
\end{conditions}

\end{definition}
It was shown by Cooper and Frieze~\cite{CF} that $Q_n = 0$ is the threshold for the existence of a giant strongly connected component under some mild conditions on the degree sequence similar to those observed in~\Cref{cond:conditions}. 

For the remainder of the paper, we shall omit the subscript $n$ for reasons of clarity - for example we write $\bd = \bd_n$ etc.
Let $\cC_k(\bd)$ be the size of the $k$th largest strongly connected component of the directed configuration model with degree sequence $\bd$.
We shall write $\cC_k$ for this quantity when the degree sequence is clear.

Our main result is the following,
\begin{theorem}\label{thm:subcrit_main_thm}
Suppose that $G(\bd)$ is a configuration model random digraph and suppose that $nQ^3(R^- R^+)^{-1} \to -\infty$.
 With high probability, there are no complex components or cycles of length $\omega(1/|Q|)$.
 Furthermore, 
 \begin{equation*}
  \bP\bbl|\cC_k| \geq \frac{\alpha}{|Q|}\bbr = 1 - \sum_{i=0}^{k-1} \frac{\xi_\alpha^i}{i!} e^{-\xi_\alpha} + o(1).
 \end{equation*}
where 
$$
\xi_{\alpha} = \int_{\alpha}^\infty \frac{e^{-x}}{x} dx.
$$
\end{theorem}
The condition $nQ^3(R^- R^+)^{-1} \to -\infty$ is likely best possible as the analogous condition in the undirected case discriminates the barely subcritical case from the critical window.
Moreover, substituting $Q, R^-$ and $R^+$ obtained from a typical degree sequence of $D(n,p)$ also enters the critical window when we have $nQ^3(R^- R^+)^{-1} \to c$ for some constant $c$. 

In prior work, \L{}uczak and Seierstad~\cite{LS} considered the model $D(n,p)$ which is a random digraph model formed by including each possible arc with probability $p$ independently.
They showed an analogous result to~\Cref{thm:subcrit_main_thm} for $p = (1-\eps)/n$ with $\eps = o(1)$ and $\eps^3n \to \infty$ in this model.
\subsection{Previous work}
The study of the giant component in random graph models was initiated by the seminal paper of Erd\H{o}s and R\'enyi~\cite{ER} regarding the giant component in $G(n,p)$.
Since then the appearance of a giant component in various models has remained an active topic of study in the area.

When working with directed graphs, there are a number of types of connected component which are of interest.
In this paper we concern ourselves with the strongly connected components.
The study of strongly connected components in random digraphs began in the model $D(n,p)$ where we include each possible edge with probability $p$ independently.
Karp~\cite{K} and \L{}uczak~\cite{L} independently showed that when $p=c/n$, then $D(n,p)$ has all strongly connected components of size $O(1)$ if $c<1$.
If instead $c>1$ they showed that there exists a unique strongly connected component of linear order with all other components of size $O(1)$.
The case $p=(1+\eps)/n$ with $\eps =o(1)$ and $|\eps|^3 n \to \infty$ was studied by \L{}uczak and Seierstad~\cite{LS}.
They showed that if $\eps^3 n \to \infty$, there is a unique strongly connected component of size $4\eps^2 n$ and other components all of size $O(1/\eps)$.
When $\eps^3n \to -\infty$ they showed an analogue of~\Cref{thm:subcrit_main_thm}:
\begin{theorem}
 Let $np = 1-\eps$ where $\eps \to 0$ but $\eps^3n \to -\infty$. Assume $\alpha>0$ is a constant and let $X_k$ denote the size of the $k$th largest strongly connected component. Then, asymptotically almost surely, $D(n,p)$ contains no complex components and
 $$
 \lim_{n\to \infty} \bP(X_k \geq \alpha/\eps) = 1- \sum_{i=0}^{k-1} \frac{\xi_\alpha^i}{i!}e^{-\xi_\alpha},
 $$
 where $\xi_\alpha = \int_\alpha^\infty \frac{e^{-x}}{x} dx$.
\end{theorem}
The so-called critical window, when $p = (1 + \lambda n^{-1/3})/n$ has also been the subject of some study.
In~\cite{MC} the author showed bounds on the size of the largest strongly connected component in this regime which are akin to bounds obtained by Nachmias and Peres for the undirected case~\cite{NP}.
Moreover, Goldschmidt and Stephenson~\cite{GS} gave a scaling limit result for the largest strongly connected components in the critical window.

The directed configuration model has also been studied previously.
It was first studied by Cooper and Frieze who showed that provided the maximum degree $\Delta < n^{1/12}/\log(n)$ then  if $Q_n<0$, there is no all strongly connected components are small and if $Q_n>0$ there is a giant strongly connected component of linear size.
The assumptions on the degree sequence have subsequently been relaxed, Graf~\cite{G} showed $\Delta < n^{1/4}$ is enough to draw the same conclusion and Cai and Perarnau~\cite{XG} improved this further to only require bounded second moments.
A scaling limit result has been recently obtained at the exact point of criticality, $Q_n = 0$ by Donderwinkel and Xie~\cite{DX} in a very closely related model where vertices' degrees are sampled from a limiting degree sequence.
\subsection{Organization}
The remainder of this paper is arranged as follows, in~\cref{sec:sg_counting} we prove some auxiliary results on subgraph counting within the directed configuration model. Then, in~\cref{sec:counting} we enumerate certain types of strongly connected directed graphs of maximum degree $4$.
In~\cref{sec:proof} we prove~\Cref{thm:subcrit_main_thm} which we break down into a few steps: first that there are no long cycles, next that there are no complex components, and finally we compute the probability the $k$th largest component has size at least $\alpha / |Q|$ via Poisson approximation.
We conclude the paper in~\cref{sec:concl} with some open questions and future work.

\section{Subgraph Bounds}\label{sec:sg_counting}
We calculate probabilities of given subgraphs in the configuration model.

Suppose that $G(\bd)$ is a configuration model random graph and let $D=(D_-,D_+)$ be the degree of a uniformly random vertex of $G(\bd)$.
We call $D$ the degree distribution of $G(\bd)$ (or of $\bd$).
Define
\begin{align*}
 \mu_{i,j} = \bE(D_-^i D_+^j), && \text{and} && \rho_{i,j} = \bE\bigg(\prod_{k=0}^{i-1}(D_--k) \prod_{\ell=0}^{j-1}(D_+-\ell)\bigg).
\end{align*}
So that the $\mu_{i,j}$ and $\rho_{i,j}$ are the moments and factorial moments of $(D_-,D_+)$ respectively.
We also define $\mu = \mu_{1,0} = \mu_{0,1}$ which is the average degree.
Furthermore, observe that $\mu_{1,1} = \frac{m}{n}(1+Q)$ which is a fact we utilise in subsequent sections.
We now state a general upper bound on the probability of finding certain subgraphs in the configuration model.

\begin{lemma}\label{lem:subgraphs_ub}
 Let $G(\bd)$ be a configuration model random digraph with degree distribution $(D_-,D_+)$. Suppose further that $G(\bd)$ has $n$ vertices and $m$ edges. Let $H$ be any digraph with $h$ vertices, $k$ edges and let $h_{i,j}$ be the number of vertices of $H$ with in-degree $i$ and out degree $j$.
 Then the probability that a uniformly random injective map $\phi:V(H) \to V(G(\bd))$ is a homomorphism is bounded above by
 \begin{equation}\label{eq:phg+}
 p^+(H,G(\bd)) := \frac{n^h}{(n)_h(m)_k} \prod_{i,j \in\bN} \rho_{i,j}^{h_{i,j}}
\end{equation}
 Furthermore, the expected number of copies of $H$ in $G(\bd)$ is bounded above by
 \begin{equation}\label{eq:ehg+}
 n^+(H,G(\bd)) := \frac{h!}{\aut(H)} \binom{n}{h} p^+(H,G(\bd)).
 \end{equation}
\end{lemma}
\begin{proof}
 Note that~\eqref{eq:ehg+} follows immediately from~\eqref{eq:phg+}.
 Thus we shall focus on the proof of~\eqref{eq:phg+}.
Let $\psi:V(H) \to V(G(\bd))$ be a fixed injective map. 
Arbitrarily order the edges of $H$ as $E_1, E_2, \ldots, E_k$ and for each edge $E_i$ define an event $\cE_i:=\{\psi(E_i) \in E(G(\bd))\}$.
Then, $\psi$ is a homomorphism if and only if every event $\cE_1, \ldots, \cE_k$ occurs.
To simplify notation, let $\cF_1 = \cE_1$ and $\cF_i = \cE_i | \cap_{j=1}^{i-1} \cE_j$ for $i \geq 2$ and note that
$$
\bigcap_{i=1}^k \cE_i = \bigcap_{i=1}^k \cF_i.
$$
Suppose that $E_i = a_ib_i$ for some $a_i,b_i\in V(H)$.
Define $s_i = |\{j<i:a_i=a_j \}|$ and $t_i = |\{j<i:b_i=b_j\}|$.
That is, $s_i$ and $t_i$ are the number of times that $a_i$ (resp. $b_i$) has previously appeared as the initial (terminal) vertex of an edge of $H$.
Also, suppose that $\psi(E_i) = a_i'b_i'$.

\begin{claim}
 $$
 \bP(\cF_i) \leq  \frac{\min(d^+(a_i')-s_i,0) \min (d^-(b_i')-t_i,0)}{m+1-i}
 $$
\end{claim}
\begin{proof}
To see this, note that if $d^+(a_i')\leq s_i$, then $\bP(\cF_i)=0$ as there are not enough stubs at $a_i'$ to create such a copy of $H$.
Similarly if $d^-(b_i')\leq t_i$, $\bP(\cF_i)=0$.

So, without loss of generality we may assume that $(d^+(a_i')-s_i), (d^-(b_i')-t_i)>0$.
Now, by definition of $\cF_i$ we have already chosen $i-1$ edges of the configuration model and we have precisely $d^+(a_i')-s_i$ out-stubs remaining at $a_i'$ and $d^-(b_i')-t_i$ at $b_i'$.
Now, consider the random matching on the $2(m+1-i)$ remaining stubs.
The probability that a specified pair of stubs is matched is $(m+1-i)^{-1}$.
There are $(d^+(a_i')-s_i)(d^-(b_i')-t_i)$ such pairs which produce $a_i'b_i'$.
Thus the expected number of edges from $a_i'$ to $b_i'$ is $(d^+(a_i')-s_i)(d^-(b_i')-t_i)/(m+1-i)$.
The claim then follows by Markov's inequality.
\end{proof}
Next we compute the probability that all of the $\cF_i$ occur simultaneously.
Due to the way in which we defined these events,
$$
\bP(\psi\text{ is a homomorphism}) = \bP\bigg(\bigcap_{i=1}^k \cF_i\bigg) = \prod_{i=1}^k \bP(\cF_i).
$$
To write down this probability succinctly, we will use the functions 
$$
f_{a,b}(x,y):= \prod_{i=1}^a \prod_{j=1}^b (x+1-i)(y+1-j)
$$
It is a simple computation to check that
\begin{equation}\label{eq:p_fixed_hom}
 \bP\bigg(\bigcap_{i=1}^k \cF_i\bigg) \leq ((m)_k)^{-1} \prod_{v \in V(H)} f_{d_H^-(v),d_H^+(v)}(d_{G(\bd)}^-(\psi(v)),d_{G(\bd)}^+(\psi(v)))
\end{equation}
Note that here we were able to remove the $\min$ function as if there are any negative contributions to the product, then there is also a contribution of value $0$ and furthermore, it is impossible for $\psi$ to be a homomorphism in this case so that~\eqref{eq:p_fixed_hom} reduces to $0 \leq 0$ in this case (which is clearly true).

To complete the proof, we extend the right hand side of equation~\eqref{eq:p_fixed_hom} to allow any function $\psi:V(H) \to V(G(\bd))$.
Choosing an uniformly random function in this way gives that the probability that a uniformly random injective function is a homomorphism is bounded above by
\begin{align}
 &\; \frac{1}{(n)_h(m)_k} \sum_{\psi:V(H) \hookrightarrow V(G(\bd))} \prod_{v \in V(H)} f_{d_H^-(v),d_H^+(v)}(d_{G(\bd)}^-(\psi(v)),d_{G(\bd)}^+(\psi(v))) \label{eq:rand_inj}\\
 \leq &\; \frac{1}{(n)_h(m)_k} \sum_{\psi:V(H) \to V(G(\bd))} \prod_{v \in V(H)} f_{d_H^-(v),d_H^+(v)}(d_{G(\bd)}^-(\psi(v)),d_{G(\bd)}^+(\psi(v)))\label{eq:rand_fun}
\end{align}
In moving from injective functions to arbitrary functions between~\eqref{eq:rand_inj} and~\eqref{eq:rand_fun}, we move from an intractable space of functions to a product space which naturally splits over the vertices of $H$.
This allows us to rewrite~\eqref{eq:rand_fun} as
\begin{align*}
& \frac{1}{(n)_h(m)_k} \sum_{\psi:V(H) \to V(G(\bd))} \prod_{v \in V(H)} f_{d_H^-(v),d_H^+(v)}(d_{G(\bd)}^-(\psi(v)),d_{G(\bd)}^+(\psi(v))) \\
= & \frac{1}{(n)_h(m)_k} \prod_{v \in V(H)} \sum_{w \in V(G(\bd))} f_{d_H^-(v),d_H^+(v)}(d_{G(\bd)}^-(w),d_{G(\bd)}^+(w)) \\
= & \frac{n^h}{(n)_h(m)_k} \prod_{v \in V(H)} \rho_{d_H^-(v),d_H^+(v)} = \frac{n^h}{(n)_h(m)_k} \prod_{i,j \in\bN} \rho_{i,j}^{h_{i,j}},
\end{align*}
which is the claimed upper bound, $p^+(H,G(\bd))$.
\end{proof}
We also need a lower bound in the case that $H$ is a cycle.
 \begin{lemma}\label{lem:cycles_lb}
  Let $G(\bd)$ be a configuration model random digraph whose degree sequence $\bd$ satisfies~\Cref{cond:conditions}. Suppose further that $G(\bd)$ has $n$ vertices and $m$ edges. Let $H$ be a directed cycle with $h$ vertices.
 Then the probability that a uniformly random injective map $\phi:V(H) \to V(G(\bd))$ is a homomorphism is bounded from below by
 \begin{equation}\label{eq:phg_cycle_lb}
 p_c^-(H,G(\bd)) := \frac{(1+Q)^h}{(n)_h} \left( 1 - \frac{2h^2\Delta^2}{\eps n} \right) \left( 1 - \frac{h \Delta^2}{2m}\right),
 \end{equation}
 where the $\eps$ in~\eqref{eq:phg_cycle_lb} comes from~\Cref{cond:conditions}
\end{lemma}
\begin{proof}
First, let us consider the probability of finding at least one edge from vertex $u$ of out-degree $a$ to vertex $v$ with in-degree $b$.
 Let $\cH_k$ be our knowledge of $G(\bd)$ up until now which consists only of a set of $k$ edges which are present.
$\cH_k$ satisfies that we have not observed any edges that affect either the out-degree of $u$ or the in-degree of $v$ and that $G(\bd)$ has $m$ edges.
So let $X$ be the number of edges between $u$ and $v$.
Then,
\begin{align*}
 \bP(X=0|\cH_k) = &\, \bbl 1-\frac{a}{m-k} \bbr\bbl 1-\frac{a}{m-k-1} \bbr \ldots \bbl 1-\frac{a}{m+1-k-b} \bbr \\
 \leq &\, \bbl 1 - \frac{a}{m} \bbr^b \\
 \leq &\, 1 - \frac{ab}{m} + \frac{a^2 b^2}{2m^2}.
\end{align*}
where the final line follows by the inequality, $(1+x)^n \leq 1-nx+\binom{n}{2} x^2$ which is valid for $n \in \bN$ and $x \geq -1$.
This allows us to deduce that the probability of seeing an edge between $u$ and $v$ is at least $\frac{ab}{m} - \frac{a^2 b^2}{2m^2}$.
Now, looking at each edge of $H$ in turn allows us to deduce that the probability $\phi$ is a homomorphism is at least
\begin{equation}\label{eq:cycle_lb_start}
 \frac{1}{(n)_h} \sum_{\phi:[h] \hookrightarrow [n]} \prod_{i=1}^h \bbl\frac{d_{\phi(i)}^- d_{\phi(i+1)}^+}{m} - \frac{(d_{\phi(i)}^- d_{\phi(i+1)}^+)^2}{2m^2} \bbr,
\end{equation}
where we consider the argument of $\phi$ modulo $h$ in~\eqref{eq:cycle_lb_start}.
Also, note that one can factorise the linear term in~\eqref{eq:cycle_lb_start} and bound the second occurrence of $d_{\phi(i)}^-d_{\phi(i+1)}^+$ by $\Delta^2$ to get the lower bound
\begin{equation}\label{eq:cycle_lb_factored}
 \frac{1}{(n)_h} \bbl 1 - \frac{ \Delta^2}{2m} \bbr \sum_{\phi:[h] \hookrightarrow [n]} \prod_{i=1}^h \frac{d_{\phi(i)}^- d_{\phi(i)}^+}{m} 
\end{equation}
The idea is to argue that we can swap the order of the product and sum in~\eqref{eq:cycle_lb_factored} without changing the result very much.
To this end, let $\Phi_i$ be the set of functions $\phi:[h]\to [n]$ such that
\begin{enumerate}[i)]
 \item $|\phi([n])| = h-i$,
 \item $d_{\phi(j)}^+ \neq 0$ for each $j \in [h]$,
 \item $d_{\phi(j)}^- \neq 0$ for each $j \in [h]$.
\end{enumerate}
Note that for any function $\phi \not \in \bigcup_{i=0}^h \Phi_i$ then $\prod_{i=1}^h d_{\phi(i)}^- d_{\phi(i)}^+ =0$.
As a result of this we observe that
\begin{equation*}
 \sum_{\phi:[h] \hookrightarrow [n]} \prod_{i=1}^h d_{\phi(i)}^- d_{\phi(i)}^+ = \sum_{\phi \in \Phi_0} \prod_{i=1}^h d_{\phi(i)}^- d_{\phi(i)}^+.
\end{equation*}
As well as the fact that
\begin{equation}\label{eq:count_cyc_eq}
  \sum_{j=0}^h\sum_{\phi \in \Phi_j} \prod_{i=1}^h d_{\phi(i)}^- d_{\phi(i)}^+ = \sum_{\phi:[h] \rightarrow [n]} \prod_{i=1}^h d_{\phi(i)}^- d_{\phi(i)}^+ =\left( \sum_{i=1}^n  d_{i}^- d_{i}^+ \right)^h= m^h(1+Q)^h.
\end{equation}
For a given function $\phi:[h] \to [n]$ we define its weight as $w(\phi):= \prod_{i=1}^h d_{\phi(i)}^- d_{\phi(i)}^+$.
We also define the weight of a set $S$ of functions in the natural way as $w(S) = \sum_{\phi \in S} w(\phi)$.
Next, we shall apply switching arguments to bound the weights of the sets $\Phi_i$ relative to one another.

Consider the auxiliary bipartite graphs, $\cG_i$ with parts $\Phi_i$ and $\Phi_{i+1}$.
We connect $\phi_i \in \Phi_i$ to $\phi_{i+1} \in \Phi_{i+1}$ by an edge in $\cG_i$ if $\phi_i$ and $\phi_{i+1}$ differ in precisely one coordinate.
For each $\phi_i \in \Phi_i$ there are at most $h^2$ ways we can change one coordinate and decrease the size of the image.
In particular we may pick any of the $h$ coordinates of $\phi_i$ and change it to $\phi_i(j)$ for some $j \in [h]$.
So, $\Delta_{\cG_i}(\Phi_i) \leq h^2$.
For each $\phi_{i+1} \in \Phi_{i+1}$, we pick any of the at least $i$ coordinates at which $\phi_{i+1}$ is not injective and choose a new image for this coordinate.
By~\Cref{cond:conditions} there are at least $\eps n/2$ ways to choose the new image,
$\delta_{\cG_i}(\Phi_{i+1}) \geq i \eps n/2$.

Combining these two results allows us to deduce that $i\eps n/2 |\Phi_{i+1}| \leq e(\cG_i) \leq h^2 |\Phi_i|$.
Upon rearrangement we find $|\Phi_i| \geq \frac{\eps n}{2h^2} |\Phi_{i+1}|$.
Note that two functions $\phi$ and $\psi$ which differ in one coordinate must also satisfy $w(\phi) \leq \Delta^2 w(\psi)$ and vice versa.
Hence $w(\Phi_i) \geq \frac{\eps n}{2h^2 \Delta^2} w(\Phi_{i+1})$.
This allows us to apply induction to deduce that 
\begin{equation*}
  \sum_{j=0}^h \sum_{\phi \in \Phi_j} \prod_{i=1}^h d_{\phi(i)}^- d_{\phi(i)}^+ =
  \sum_{j=0}^h w(\Phi_j) \leq w(\Phi_0) \sum_{j=0}^h \left( \frac{2h^2\Delta^2}{\eps n} \right)^j \leq \frac{w(\Phi_0)}{1 - \frac{2h^2\Delta^2}{\eps n}}.
\end{equation*}
So by~\eqref{eq:count_cyc_eq}, $w(\Phi_0) \geq (1 - \frac{2h^2\Delta^2}{\eps n}) m^h(1+Q)^h$.
Combining this with~\eqref{eq:cycle_lb_factored} allows us to deduce the statement of the lemma, that the following is a lower bound on the probability of finding a cycle at a specified position in $G(\bd)$:
$$
 \frac{(1+Q)^h}{(n)_h} \left( 1 - \frac{2h^2\Delta^2}{\eps n} \right) \left( 1 - \frac{h \Delta^2}{2m}\right).
$$

\end{proof}
\section{Digraph Counting}\label{sec:counting}
In this section we will prove the following bound on the number of strongly connected multi-digraphs with maximum total degree $4$.
This is similar to~\cite[Lemma 2.3]{MC} although a weaker bound suffices here allowing us to simplify the proof somewhat.

\begin{lemma}\label{lem:count_scdegseq}
 Suppose that $n,m,a,b \in \bN$ such that $n+a+b=m$.
 Let $N(n,a,b)$ be the number of labelled strongly connected multi-digraphs with $n$ vertices and degree distribution given by
\begin{center}
\begin{tabular}{ c|c|c }
 in-degree & out-degree & quantity \\ \hline
  $1$ & $1$ & $n-2a-b$ \\ \hline
  $1$ & $2$ & a \\ \hline
  $2$ & $1$ & a \\ \hline
  $2$ & $2$ & b
\end{tabular}
\end{center}
Then, we have the following bound,
\begin{equation*}
 N(n,a,b) \leq (3a+2b) (m-1)! \binom{n}{a,a,b}
\end{equation*}
\end{lemma}
To prove this bound we will use the preheart configuration model of P\'erez-Gim\'enez and Wormald~\cite{PGW} which we shall define as follows.

A \emph{preheart} is a multi-digraph with minimum semi-degree at least $1$ and no cycle components.
The $\emph{heart}$ of a preheart $D$ is the multidigraph $H(D)$ formed by suppressing all vertices of $D$ which have in and out degree precisely $1$.
For a degree sequence $\bd$ on vertex set $V$, define 
$$
T = T(\bd) = \{ v \in V : d^+(v) +d^-(v) \geq 3 \}.
$$
To form the \emph{preheart configuration model}, first we apply the configuration model to $T$ to produce a heart $H$.
Given a heart configuration $H$, we construct a preheart configuration $Q$ by assigning $V \sm T$ to $E(H)$ such that the vertices assigned to each arc of $H$ are given a linear order.
Denote this assignment including the orderings by $q$.
Then the preheart configuration model, $\cQ(\bd)$ is the probability space of random preheart configurations formed by choosing $H$ and $q$ uniformly at random.

It is easy to see that every strongly connected digraph is produced by the preheart configuration model.
Thus counting the number of possible outcomes from the preheart configuration model gives an upper bound for the number of strongly connected digraphs with the same degree sequence.
We now count the number of preheart configurations using~\cite[Lemma 2.4]{MC}.
\begin{lemma}
\label{lem:preheartconfigs}
In the preheart configuration model with $n$ vertices, $m$ edges and degree sequence $\bd$, let $n' = |T(\bd)|$ be the number of vertices of the heart.
 Then there are a total of
 $$
 \frac{n'+m-n}{m} m!
 $$
 preheart configurations.
\end{lemma}
From this lemma, we may prove~\Cref{lem:count_scdegseq}.
\begin{proof}[Proof of~\Cref{lem:count_scdegseq}]
 First, we choose the degree sequence.
 So note that there are at most $\binom{n}{a,a,b}$ ways in which we can give $a$ vertices in-degree $1$ and out-degree $2$, $a$ vertices in-degree $2$ and out-degree $1$, $b$ vertices in-degree $2$ and out-degree $2$ and the remainder in-degree $1$ and out-degree $1$.
 Having fixed this degree sequence, by~\Cref{lem:preheartconfigs} as the heart contains $2a+b$ vertices and there are $a+b$ more edges than vertices, the number of strongly connected digraphs with this degree sequence is bounded above by $\frac{3a+2b}{m} m!$.
 Hence the number of strongly connected digraphs with degree distribution as in the statement of the lemma is at most
 \begin{equation*}
  (3a+2b) (m-1)! \binom{n}{a,a,b}
 \end{equation*}
as claimed.
\end{proof}

\section{Proof of~\Cref{thm:subcrit_main_thm}}\label{sec:proof}
In this section we prove~\Cref{thm:subcrit_main_thm} and show that every strongly connected component is a cycle and that these cycles are not particularly large.
In particular this provides a configuration model analogue of~\cite[Theorem 7]{LS}.
Working with the configuration model introduces several additional difficulties with the proof of this, foremost of these being that we do not have enough control of the subgraph counts to compute the factorial moments and show that they converge to those of a Poisson distribution.
We will instead use the Chen-Stein method for Poisson approximation which only requires good control of the first moment and an upper bound on the second.

The proof of this theorem splits naturally into four parts.
Choose functions $f(n) \gg g(n)$ which are defined such that $f(n) = \omega(\sqrt{m/|Q|})$, $f(n) = o(m|Q|/R^-)$ and $g(n) = \omega(1/|Q|)$.
Moreover for this section we shall assume that $R^- \geq R^+$ and if this is not the case, we swap the orientations of all edges to get an equivalent digraph with $R^- \geq R^+$ as desired.
We will say that a cycle $C$ is
\begin{itemize}
 \item \emph{Long} if $|C| \geq f(n)$,
 \item \emph{Medium} if $g(n) < |C| < f(n)$,
 \item \emph{Short} if $|C| \leq g(n)$.
\end{itemize}
First we will show that there are no long or medium cycles in the directed configuration model.
Next, we show that there are no complex components and finally we show the result on the distribution of the length of the $k$th longest cycle.
\subsection{Long Cycles}
\begin{lemma}\label{lem:long_cyc}
 The configuration model, $G(\bd)$ has no long cycles.
\end{lemma}
To show that there are no long cycles, it suffices to show that the out-component of an arbitrary vertex is bounded above by $f(n)$.
Certainly, the longest cycle in a directed graph is at most the size of the largest out-component and so the lemma follows.

Thus, consider the following version of the branching process of Hatami and Molloy~\cite{HM} for the out-component in a digraph.
For a vertex $v$ we explore its out-component in the random digraph $G(\bd)$ as follows.
We will have a partial subdigraph $C_t$ at time $t$ consisting of the vertices explored thus far.
$C_t$ will consist of all in- and out-stubs of some vertices of $G(\bd)$ together with a matching of some of the stubs.
If there are unmatched out-stubs in $C_t$ we will pick one at random and match it to some in-stub which yields an edge of $C_t$.
We define $Y_t$ as the number of unmatched out-stubs in $C_t$.
Thus, $Y_t = 0$ indicates that we have explored  an out-component in its entirety.
Formally we define the exploration process as follows.
\begin{itemize}
 \item Choose a vertex $v$ and initialise $C_0 = \{v\}$ and $Y_0 = d^+(v)$.
 \item While $Y_t>0$, choose an arbitrary unmatched out-stub of any vertex $v \in C_t$. Pick a uniformly random unmatched in-stub and let $u$ be the vertex to which this in-stub belongs. Match these two stubs forming an edge of $G$.
 \begin{itemize}
 \item If $u \not\in C_t$ we add it so $C_{t+1} = C_t \cup \{u\}$ and $Y_{t+1} = Y_t + d^+(u)-1$.
 \item Otherwise, $C_{t+1} = C_t$, $Y_{t+1} = Y_t -1$.
 \end{itemize}
\end{itemize}
Note that $Y_t$ does not depend on how we have exposed $C_t$ so $C_t$ and $Y_t$ are Markov processes.
We define the following quantities.
\begin{itemize}
 \item $D_t := Y_t + \sum_{u \not\in C_t} d^+(u)$, the number of unmatched out-stubs at time $t$. Note this is also the number of unmatched in-stubs at time $t$.
 \item $v_t := \emptyset$ if $C_{t-1}$ and $C_t$ have the same vertex set. Otherwise it is the unique vertex in $C_t\sm C_{t-1}$.
 \item $Q_t := \frac{\sum_{u \not\in C_t} d^-(u) d^+(u)}{D_t} - 1$.
\end{itemize}
Note that initially $Q_t = Q$.
Also, for unvisited vertices $u \not\in C_t$, the probability that we explore $u$ next is $\bP(v_{t+1} = u) = \frac{d^-(u)}{D_t}$.
Hence, provided that $Y_t>0$, the expected change in $Y_t$ is
\begin{equation}\label{eq:expected_change_EP}
 \bE(Y_{t+1}-Y_t | C_t) = \sum_{u \not\in C_t} \bP(v_{t+1} = u) (d^+(u) - 1) = \frac{\sum_{u\not\in C_t} d^-(u)d^+(u)}{D_t} - 1 = Q_t.
\end{equation}
As long as $Q_t$ remains close to $Q$ we expect that $Y_t$ is a random walk with drift approximately $Q$.
So in particular, for our setting of $Q<0$, we expect that the random walk will quickly return to 0.
Thus, we shall start by showing that the drift parameter $Q_t$ is indeed close to $Q$ with high probability.
For this we shall use the following formulation of the Azuma-Hoeffding inequality (see~\cite[Theorem 2.25]{JLR}).
\begin{theorem}[Azuma-Hoeffding Inequality]\label{thm:AH_ineq}
Let $(X_k)_{k=0}^n$ be a martingale with $X = X_n$, $X_0 = \bE(X)$.
Suppose there exist constants, $c_k >0$ such that
$$
|X_k - X_{k-1}| \leq c_k\; \text{ for all }\; k \leq n.
$$
Then for any $\lambda \geq 0$,
\begin{equation*}
 \bP(|X_n - X_0| \geq \lambda) \leq 2 \exp\left(-\frac{\lambda^2}{2\sum_{k=1}^n c_k^2}\right).
\end{equation*}
\end{theorem}
For $t \geq 1$ define $W_t:=Q_t - Q_{t-1} -\bE(Q_t - Q_{t-1}|C_{t-1})$.
Also, we define $X_0 = Q$ and for $t\geq 1$ let
\begin{equation}\label{eq:Xt_defn}
 X_t := X_0 + \sum_{i=1}^t W_i = Q_t - \sum_{i=1}^t \bE(Q_i - Q_{i-1}|C_{i-1}).
\end{equation}
It is a simple check that the $X_t$ forms a martingale.
Furthermore, $|Q_t-Q| \leq |Q_t-X_t| + |X_t-Q|$ so to bound the probability that $|Q_t-Q|$ is large, we show that $|Q_t - X_t|$ is small and bound the probability that $|X_t - Q|$ is large.

For the second of these, consider the auxiliary random variables
$$
\widetilde{Q}_t := \frac{\sum_{u \not\in C_t} d^-(u) d^+(u)}{D_{t-1}} - 1.
$$
That is we change $Q_t$ to have the same denominator as $Q_{t-1}$.
As we assume that $t \leq m/2$, $D_{t} \geq m/2$ so combining this with the fact that $|D_t - D_{t-1}| \leq 1$ we deduce that
\begin{equation}\label{eq:QvsQtwiddle}
 |Q_t - \widetilde{Q}_t| = \frac{|D_t - D_{t-1}| \sum_{u \not\in C_t} d^-(u) d^+(u)}{D_t D_{t-1}} \leq \frac{4}{m}.
\end{equation}
Also, as there is at most one vertex whose contributions are removed in moving from $Q_{t-1}$ to $Q_t$, then
\begin{equation}\label{eq:Qtminus1vst}
|Q_{t-1} - \widetilde{Q}_t| \leq \frac{4\Delta^2}{m}.
\end{equation}
Combining equations~\eqref{eq:QvsQtwiddle} and~\eqref{eq:Qtminus1vst} gives an upper bound on $|Q_t - Q_{t-1}|$ which we may then use to bound the martingale differences almost surely,
\begin{equation}\label{eq:X_diff_bound}
 |X_t - X_{t-1}| \leq \frac{8\Delta^2 + 8}{m} \leq \frac{16 \Delta^2}{m}.
\end{equation}
Next, we will bound the terms $\bE(Q_t - Q_{t-1}|C_{t-1})$.
It will be convenient to do this in two stages utilising the auxiliary random variables $\widetilde{Q}_t$.
So, first note that 
\begin{align}
0\leq \bE(Q_{t-1} - \widetilde{Q}_t|C_{t-1}) & = \sum_{u \not\in C_{t-1}} \bP(v_t = u) \frac{d^-(u)d^+(u)}{D_{t-1}} \leq \sum_{u \in V(G)} \frac{(d^-(u))^2d^+(u)}{D_{t-1}^2} \nonumber \\
& \leq \sum_{u \in V(G)} \frac{4(d^-(u))^2d^+(u)}{m^2} \leq \frac{4R^- + 4}{m}. \label{eq:EQtwidQ}
\end{align}
We can combine~\eqref{eq:EQtwidQ} with~\eqref{eq:QvsQtwiddle} and apply~\Cref{cond:conditions} to deduce that
\begin{equation}
 |\bE(Q_{t} - Q_{t-1}|C_{t-1})| \leq \frac{4R^- + 8}{m} \leq \frac{12R^-}{\zeta m}. \label{eq:QvsQ_prev}
\end{equation}
This leaves us in a situation in which we can compare $X_t$ and $Q_t$,
\begin{equation}
 |X_t - Q_t| \leq \sum_{i=1}^t |\bE(Q_{i} - Q_{i-1}|C_{i-1})| \leq \frac{12R^- t}{\zeta m}. \label{eq:XtQt_bound}
\end{equation}
So provided that $t\leq \frac{m\zeta|Q|}{48 R^-}$ we have $|X_t - Q_t| \leq |Q|/4$ and in particular, for any such $t$,
\begin{align*}
 &\bP\left(Q_t - Q \geq \frac{|Q|}{2}\right) = \bP\left(Q_t-X_t+X_t-Q \geq \frac{|Q|}{2}\right) \\
 \leq\;& \bP\left(X_t - Q \geq \frac{|Q|}{4} \right) \leq \bP\left(|X_t - Q| \geq \frac{|Q|}{4} \right)
\end{align*}
Whereupon we can apply the Azuma-Hoeffding inequality as $X_0 = Q$.
We can use the bound from~\eqref{eq:X_diff_bound} for the $c_k$.
Substituting into~\Cref{thm:AH_ineq} yields,
\begin{equation}
 \bP\left(|X_t - Q| \geq \frac{|Q|}{4} \right) \leq \exp \left( -\frac{|Q|^2 m^2}{8192 t \Delta^4}\right) \leq \exp \left( -C |Q|m n^{-2/3} \log(n) \right), \label{eq:apply_AH}
\end{equation}
where the second inequality in~\eqref{eq:apply_AH} comes from $t \leq \frac{m\zeta|Q|}{48 R^-}$, $\Delta \leq n^{1/6} \log^{-1/4}(n)$ and $R^- \geq \zeta$.
We could improve the dependence on $\Delta$ by using Freedman's inequality~\cite{Freedman} in place of the Azuma-Hoeffding inequality here however there are other points where we require $\Delta \leq n^{1/6}$ and so this would only remove the $\log^{-1/4}(n)$ term in~\Cref{cond:conditions}.
Note $mn^{-2/3} \geq m^{1/3}/2$ hence $|Q|mn^{-2/3} \to \infty$ and so for any large enough $n$,
\begin{equation}
 \bP\left(Q_t - Q \geq \frac{|Q|}{2}\right) = \bP\left( Q_t \geq -\frac{|Q|}{2}\right) \leq n^{-2}. \label{eq:final_EP_bound}
\end{equation}
Now that we have shown that $Q_t$ is concentrated around $Q$, we can proceed to show that $G(\bd)$ has no large components with high probability via a stopping time argument.
We will use the following version of Doob's optional stopping theorem~\cite[Theorem 10.10]{W}
\begin{theorem}[Optional Stopping Theorem]
\label{thm:DOST}
Let $X$ be a supermartingale and let $\tau$ be a stopping time. Then $X_\tau$ is integrable and furthermore,
\begin{equation*}
 \bE(X_\tau) \leq \bE(X_0)
\end{equation*}
whenever $\tau$ is bounded.
\end{theorem}
So, now let us show that there is no component of size larger than $f(n)$.
For each $v \in V(G)$ we shall consider the exploration process started at $v$.
Recall that $f(n) = o(n|Q|/R^-)$ and so for sufficiently large $n$, $f(n) \leq \frac{m\zeta|Q|}{48 R^-}$.
Hence we have $Q_t \leq -|Q|/2$ with high probability for each $t \leq f(n)$.
Define the stopping time
$$
\tau := \min\{t \geq 0 | Y_t = 0 \text{ or } Q_t \geq -|Q|/2 \text{ or } t=f(n) \}
$$
Recall that for $t \leq \frac{m\zeta|Q|}{48 R^-}$ we have $\bE(Y_t - Y_{t-1}) = Q_{t-1} \leq -|Q|/2$. Thus, $Y_{\min(t,\tau)} + |Q| \min(t,\tau)/2$ is a supermartingale.
Clearly, $\tau$ is bounded by $f(n)$ so we may apply~\Cref{thm:DOST} to $Y_{\min(t,\tau)} + |Q| \min(t,\tau)/2$ from which we deduce that
\begin{equation*}
 \bE\left(Y_\tau + \frac{|Q|\tau}{2}\right) \leq Y_0 = d^+(v).
\end{equation*}
Upon rearrangement this yields,
\begin{equation*}
\bE(\tau) \leq 2\frac{d^+(v)-\bE(Y_\tau)}{|Q|} \leq \frac{2 d^+(v)}{|Q|}.
\end{equation*}
By Markov's inequality we can deduce
$$
\bP(\tau = f(n)) \leq \frac{2 d^+(v)}{|Q|f(n)}
$$
The only other way in which we could have $Y_\tau \neq 0$ is if for some $i$ we have $Q_i \geq -|Q|/2$.
A union bound allows us to deduce that this occurs with probability at most $f(n) n^{-2} \leq n^{-1}$.
So for any large enough $n$,
$$
\bP(Y_\tau \neq 0) \leq \frac{2 d^+(v)}{|Q|f(n)} + \frac{1}{n}.
$$
Define $Z$ as the number of vertices of $G$ which lie in cycles of size at least $f(n)$.
Note that any such vertex must have out component of size at least $f(n)$.
Thus,
\begin{align*}
 \bP(|\cC_1|\geq f(n)) & 
 = \bP(Z \geq f(n))
 \leq \frac{\bE(Z)}{f(n)} 
 \leq \frac{1}{f(n)} \sum_{v \in V(G)} \bP(|C^+(v)| \geq f(n)) \\
 & \leq \frac{1}{f(n)} \sum_{v \in V(G)} \left( \frac{2 d^+(v)}{|Q|f(n)} + \frac{1}{n} \right)
 = \frac{m}{|Q|f(n)^2} + \frac{1}{f(n)} = o(1)
\end{align*}
as $f(n) \gg \sqrt{m/|Q|}$. 
Thus there are no long cycles in $G(\bd)$.
\subsection{Medium Cycles}
\begin{lemma}\label{lem:med_cyc}
The configuration model, $G(\bd)$ has no medium cycles.
\end{lemma}

Our next step is to apply~\Cref{lem:subgraphs_ub} to show there are no medium cycles.
By~\Cref{lem:subgraphs_ub}, the probability that $G(\bd)$ has a cycle of length $h$ in any particular location is at most $\frac{n^h \mu_{1,1}^h}{(n)_h(m)_h} $.
Thus the expected number of cycles of length $h$ in $G(\bd)$ is at most
\begin{align}
\frac{h!}{|\aut(\overrightarrow{C_h})|} \binom{n}{h} \frac{n^h \mu_{1,1}^h}{(n)_h(m)_h} & = \frac{m^h}{(m)_h}\frac{(1+Q)^h}{h}
\leq \bbl 1+\frac{h}{m-h} \bbr^h  \frac{(1+Q)^h}{h} \nonumber \\
& \leq \bbl 1+\frac{2h}{m} \bbr^h  \frac{(1+Q)^h}{h} 
\leq \frac{e^{hQ + \frac{2h^2}{m}}}{h}. \label{eq:med_cyc_ub}
\end{align}
For any $g(n) \leq h \leq f(n)$, we have $2h^2/m \leq 4h^2/n \leq h|Q|/2$ as $f(n) = o(n|Q|)$.
Then, the expected number of cycles of length between $g(n)$ and $f(n)$ is at most
\begin{equation}\label{eq:cycle_expectation}
 \sum_{h=g(n)}^{f(n)}  \frac{e^{hQ+ \frac{2h^2}{m}}}{h}
 \leq \sum_{h=g(n)}^{f(n)}  \frac{e^{\frac{hQ}{2}}}{h}
 \leq \int_{g(n)}^\infty \frac{e^{\frac{hQ}{2}}}{h} dh
 = \int_{-\frac{Qg(n)}{2}}^\infty \frac{e^{-\lambda}}{\lambda} d\lambda
 = E_1\bbl -\frac{Qg(n)}{2} \bbr,
\end{equation}
where $E_1(x)$ is the exponential integral function and the first equality follows by making the substitution $\lambda = -Qh/2$ (recall that $Q<0$ and so this substitution preserves positivity).
It is straightforward to bound $0 \leq E_1(x) \leq e^{-x}/x$ which allows us to conclude that $E_1(x) \to 0$ as $x \to \infty$.
Note that $g(n) = \omega(1/|Q|)$ and so $-Qg(n) \to \infty$.
Thus the expected number of cycles in $G(\bd)$ of length at least $g(n)$ is $o(1)$.
So by Markov's inequality, there are no such cycles with high probability.

\subsection{Complex Components}
\begin{lemma}\label{lem:com_com}
The configuration model, $G(\bd)$ has no complex components.
\end{lemma}

We begin by defining digraphs $S(a,b,c)$ and $T(a,b)$ for $a,b,c \in \bN$.
Let $S(a,b,c)$ be the digraph with $a+b+c-1$ vertices consisting of vertices $u$, $v$ and three internally disjoint paths.
One of length $a$ from $u$ to $v$, one of length $b$ from $u$ to $v$ and one of length $c$ from $v$ to $u$.
Let $T(a,b)$ be the digraph with $a+b-1$ vertices consisting of two cycles, one of length $a$, one of length $b$ which intersect at a single vertex, $u$.
We can use the ear decomposition of a strongly connected digraph to deduce that if $G(\bd)$ contains any complex components, then it contains a subgraph which is either a copy of $S(a,b,c)$ or a copy of $T(a,b)$.
Note that both of these are the union of  two cycles and and by the results of the previous two sections, there are no cycles with more than $g(n)$ vertices with high probability.
Thus we only need to show there are none of these motifs on at most $2g(n)$ vertices to deduce that there are none in $G(\bd)$.
 
 Unlike \L{}uczak and Seierstad~\cite{LS}, we must treat these cases separately as in the first case, we have two vertices of total degree $3$ and the rest of total degree $2$ and in the second there is one vertex of total degree $4$ in place of the total degree $3$ vertices which changes the result of applying~\Cref{lem:subgraphs_ub}.
 
 First let us consider $S(a,b,c)$. There are at most $h^2 h!$ ways of finding such subgraphs on $h$ vertices ($\leq h^2$ ways of choosing path lengths connecting the two degree $3$ vertices and assuming the associated automorphism groups are all trivial gives this bound).
 Thus, we may apply~\Cref{lem:subgraphs_ub} to deduce that the expected number of such subgraphs in the configuration model with parameters as in the statement of~\Cref{thm:subcrit_main_thm} is at most
 \begin{equation}\label{eq:complex_expectation_2deg3}
  \sum_{h=1}^{2g(n)} \frac{n^h}{(m)_{h+1}} h^2 h! \mu_{1,1}^{h-2} \rho_{1,2} \rho_{2,1}
  = \frac{R^- R^+}{m} \sum_{h=1}^{2g(n)} \frac{m^{h+1}}{(m)_{h+1}} h^2 (1+Q)^{h-2}
  \leq \frac{R^- R^+}{m} \int_0^{2g(n)} x^2 e^{\frac{xQ}{2}} dx,
 \end{equation} 
 where we eliminate the term $\frac{m^{k+1}}{(m)_{k+1}}$ in the above in the same way as in~\eqref{eq:med_cyc_ub}.
 An integral of the form seen in~\eqref{eq:complex_expectation_2deg3} can be evaluated by integrating by parts twice to deduce the following (where $t>0$),
 \begin{equation*}
  \int_0^y x^2 e^{-tx} = \frac{2}{t^3} - e^{-ty} \bbl \frac{2}{t^3}+\frac{2y}{t^2} + \frac{2y^2}{t} \bbr \leq \frac{2}{t^3}.
 \end{equation*}
Thus the expected number of these subgraphs in $G(\bd)$ can be bounded above by $\frac{16 R^- R^+}{m|Q|^3} \to 0$.
 
 The second case is $T(a,b)$ where we have a degree $4$ vertex.
 In this case there are at most $h h!$ such subgraphs on $h$ vertices ($\leq h$ choices of the two cycle lengths and assuming the associated automorphism groups are all trivial gives this bound).
Again we apply~\Cref{lem:subgraphs_ub} to compute the expected number of such subgraphs of size at most $2g(n)$ which this time is at most
  \begin{equation}\label{eq:complex_expectation_deg4}
  \sum_{h=1}^{2g(n)} \frac{n^h}{(m)_{h+1}} h \mu_{1,1}^{h-1}\rho_{2,2} 
  \leq \frac{R^+ \Delta}{m} \sum_{h=1}^{2g(n)} \frac{m^{h+1}}{(m)_{h+1}}h e^{-hQ} 
  \leq \frac{R^+\Delta}{m} \int_0^{2g(n)} x e^{-\frac{xQ}{2}} dx 
 \end{equation}
 Note we may pick either $\rho_{2,2,} \leq \frac{m}{n} R^-\Delta$ or $\frac{m}{n} R^+\Delta$ here by selecting which part of the product $d_i^-(d_i^--1)d_i^+(d_i^+-1)$ to bound by $\Delta$ in computing $\rho_{2,2}$ and so we pick the smaller of the two.
We may proceed similarly to before, integrating by parts which allows us to bound integrals of the form found in~\eqref{eq:complex_expectation_deg4} as
\begin{equation*}
  \int_0^y x e^{-tx} = \frac{1}{t^2} - e^{-ty} \bbl \frac{1}{t^2}+\frac{y}{t}  \bbr \leq \frac{1}{t^2}.
 \end{equation*}
So, we can bound~\eqref{eq:complex_expectation_deg4} above by $\frac{4 R^+\Delta}{m |Q|^2} \leq \frac{4 (R^+ R^-)^{2/3}}{m^{2/3} \zeta^{1/3}|Q|^2}\frac{\Delta}{m^{1/3}}  \to 0$.
Thus by Markov's inequality there are no copies of $S(a,b,c)$ or $T(a,b)$ on at most $2g(n)$ vertices in $G(\bd)$ with high probability.
Combining~\Cref{lem:long_cyc} and~\Cref{lem:med_cyc} we deduce there are no cycles of length at least $g(n)$ with high probability.
As any complex strongly connected digraph contains a copy of at least one of these, we deduce that $G(\bd)$ contains no complex components with high probability.


%

\subsection{Length of the kth largest cycle}
\begin{lemma}\label{lem:kth_cyc}
The $kth$ longest cycle in the configuration model, $G(\bd)$ follows the distribution given in~\Cref{thm:subcrit_main_thm}
\end{lemma}
The idea now will be to apply a local coupling version of the Chen-Stein method (see~\cite[Theorem 2.8]{TE}) to deduce that the number of cycles in $G(\bd)$ of length between $\alpha/|Q|$ and $g(n)$ converges to a Poisson distribution of mean $\xi_\alpha$.
Let us start by stating the version of the Chen-Stein method which we will apply,
\begin{theorem}\label{thm:localcouplingstein}
Let $W = \sum_{i \in \Gamma} X_i$ be a sum of indicator variables and let $p_i:=\bE(X_i)$.
For each $i \in \Gamma$, divide $\Gamma \sm \{i\}$ into two sets, $\Gamma_i^s$ and $\Gamma_i^w$.
Define
\begin{equation*}
 Z_i := \sum_{j \in \Gamma_i^s} X_j\; \text{ and }\; W_i := \sum_{j \in \Gamma_i^w} X_j.
\end{equation*}
Suppose that there exist random variables, $W_i^1$ and $\widetilde{W}_i^1$ defined on the same probability space such that
\begin{equation*}
 \cL(\widetilde{W}_i^1) = \cL(W_i|X_i=1)\; \text{ and }\; \cL(W_i^1) = \cL(W_i).
\end{equation*}
Then,
\begin{equation}\label{eq:stein_bound}
 \dtv(W,Po(\bE(W))) \leq \min(1,\bE(W)^{-1}) \sum_{i \in \Gamma} \left( p_i \bE(X_i+Z_i) + \bE(X_iZ_i) + p_i \bE|\widetilde{W}_i^1-W_i^1| \right)
\end{equation}
\end{theorem}
Note that this lemma requires us to have a copy of $W_i|X_i = 1$.
To create such a copy, we will use the following lemma to couple the configuration model with itself conditioned on the containment of a given subgraph.
\begin{lemma}\label{lem:bipswitching}
 Let $G=(A \cup B, E)$ be a balanced bipartite complete graph and let $\cM$ be a uniformly chosen random perfect matching of $G$.
 Suppose that $a_1, a_2, \ldots, a_k$ are distinct elements of $A$ and $b_1, b_2, \ldots, b_k$ are distinct elements of $B$.
 Then, the following procedure gives a copy of $\cM|(a_1b_1, a_2b_2, \ldots, a_kb_k \in \cM)$:
 \begin{enumerate}
  \item Sample an element $M$ from $\cM$ and set $M_0=M$.
  \item For each $i \leq k$:
  \begin{itemize}
   \item If $a_ib_i$ is an edge of $M_{i-1}$ set $M_i = M_{i-1}$.
   \item Otherwise, let $a_i'$ be the unique neighbour of $b_i$ and $b_i'$ be the unique neighbour of $a_i$ in $M_{i-1}$. Let $M_i = (M_{i-1} \sm \{a_ib_i', a_i'b_i \}) \cup \{a_ib_i, a_i'b_i' \}$.
  \end{itemize}
 \end{enumerate}
 That is $M_k$ is sampled uniformly from $\cM|(a_1b_1, a_2b_2, \ldots, a_kb_k \in \cM)$.
\end{lemma}
\begin{proof}
Note that it is sufficient to prove the lemma for $k=1$ as if $\cB =\cM|(a_1b_1, a_2b_2, \ldots, a_{k-1}b_{k-1} \in \cM)$ then clearly, $\cB|(a_kb_k \in \cB) = \cM|(a_1b_1, a_2b_2, \ldots, a_kb_k \in \cM)$.
So, the general case follows by induction on the $k=1$ case.

Now, we prove the lemma for $k=1$. To do so, we will show that each of the $(n-1)!$ atoms of $\cM|(a_1b_1 \in \cM)$ comes from precisely $n$ atoms of $\cM$ via this switching approach.
So let $M$ be an atom of $\cM|(a_1b_1 \in \cM)$ and let $e = cd$ be an edge of $M$ with $c\in A, d\in B$.
Now, consider the inverse switching, $M \to (M \sm \{ab, cd \}) \cup \{ad,bc \}$ (note if we chose the edge $ab$, then this is simply the identity $M \to M$).
Thus, for each atom $M$ of $\cM|(a_1b_1 \in \cM)$ there are $n$ ways to get to an atom of $\cM$.
Similarly we can show that each atom of $\cM$ is mapped to a unique element of $\cM|(a_1b_1 \in \cM)$.
Thus, as $\cM$ has the uniform distribution, so does the random variable obtained by our procedure above.
\end{proof}
Note that this lemma does not allow us to generate the configuration model conditioned on the existence of a subgraph specified in the usual way by the locations of its vertices unless all of the involved vertices have in- and out-degrees at most $1$.
This is due to the fact that~\Cref{lem:bipswitching} allows us to condition on which pairs of stubs are connected rather than which pairs of vertices.
Instead we shall condition on the existence of \emph{principal subgraphs} which we shall define as follows.
\begin{definition}
 Let $G(\bd)$ be a configuration model random digraph with degree distribution $\bd$.
 Let $M = M(G(\bd))$ be the associated perfect matching of in- and out-stubs.
 A \emph{principal subgraph} of $G(\bd)$ is an event of the form $M' \subseteq M$ where $M'$ is a partial matching of in- and out-stubs.
\end{definition}
A similar notion was seen in~\cite{LuSz} in which their canonical events are precisely the same as our principal subgraphs.
Furthermore, note that the event corresponding to a subgraph in $G(\bd)$ may be written as a union of events corresponding to principal subgraphs.
In particular, the existence of the cycle $v_1v_2\ldots v_k$ in $G(\bd)$ can be written as the union of $\prod_{i=1}^k d_G^-(v_i) d_G^+(v_i)$ events corresponding to principal cycles.
For this reason, in places where there may be some ambiguity as to whether we are dealing with principal subgraphs or not, we shall refer to subgraphs in the usual sense as \emph{union subgraphs}. 

This leaves us in a setting to which we can apply~\Cref{thm:localcouplingstein}.
We shall show that the number of cycles in $G(\bd)$ of lengths between $\alpha/|Q|$ and $g(n)$ is Poisson distributed with mean $\xi_\alpha$.
Let $\Gamma$ be the set of all principal cycles which have lengths between $\alpha/|Q|$ and $g(n)$.
For each $C \in \Gamma$ and digraph $J$ on the same vertex set and stubs as $G(\bd)$ let $X_C(J)$ be the indicator function that $C$ is a principal subgraph of $J$.
Also, we split $\Gamma \sm \{C\}$ into a set strongly dependent on $C$ and a set weakly dependent on $C$.
Define the strongly dependent set $\Gamma_C^s$ to be the set of all principal cycles which share at least one vertex with $C$.
The weakly dependent set $\Gamma_C^w$ contains all of the other principal cycles, it is the set of principal cycles which are vertex disjoint from $C$.
Finally, we define $W_C^1$ to be $W_C$ for an independent copy $G'$ of $G(\bd)$ and $\widetilde{G}_C$ to be obtained from $G'$ by applying a $4$-cycle switching to each edge of $C$ in $G_C$ in turn and define $\widetilde{W}_C^1$ in the obvious way to be $W_C$ for the digraph $\widetilde{G}_C$.
That is,
$$
W_C = \sum_{C' \in \Gamma_C^w} X_{C'}(G') \qquad\qquad\qquad \widetilde{W}_C^1 = \sum_{C' \in \Gamma_C^w} X_{C'}(\widetilde{G}_C)
$$
These variables clearly satisfy the assumptions of~\Cref{thm:localcouplingstein} and so we must bound the expectations in the statement of the theorem to compute an upper bound on $\dtv(W,Po(\bE(W)))$.
The idea now is to reduce the whole problem to one of bounding the expected numbers of certain subgraphs being contained in $G(\bd)$, $G'$ and $\widetilde{G}_C$.
For the remainder of this section, we write $X_C$ for $X_C(G)$ and $p_C = \bE(X_C)$ unless specified otherwise.

We shall bound the three terms from~\eqref{eq:stein_bound} one by one.
First let us consider the term
\begin{equation}\label{eq:eta_defn}
\eta := \sum_{C \in \Gamma} p_C \bE(X_C + Z_C) = \sum_{C \in \Gamma} \sum_{C' \in \Gamma_C^s \cup \{ C \}} p_C p_C'
\end{equation}
Note that the set, $\{C' \in \Gamma_C^s \cup \{ C \}\}$ is the set of all cycles which share at least one vertex with $C$.
Thus, the union of $C$ and $C'$ is a strongly connected digraph (here we allow multiple edges).
In the computation of $\eta$ we will use the following generalisation of~\cite[Lemma 3.4]{MC}.
No changes are required to the proof of the lemma as presented in~\cite{MC} to generalise from digraphs to multi-digraphs.
\begin{lemma}
\label{lem:27k_cycle_union}
Each strongly connected multi-digraph $D$ with excess $k$ may be formed in at most $27^k$ ways as the union of a pair of directed cycles $C_1$ and $C_2$.
\end{lemma}
Define $\Xi$ to be the set of all strongly connected multi-digraphs with vertices a subset of $V(G(\bd))$ such that all vertices have degrees $d^+(v)=d^-(v)=1$ or $d^+(v)=d^-(v)=2$ with at least one vertex which has $d^+(v)=d^-(v)=2$.
Note that $\Xi$ is precisely the set of multi-digraphs which can be formed as the edge disjoint union of the two cycles $C \in \Gamma$ and $C' \in \Gamma_{C}^s$.
Furthermore, note that the excess of a multi-digraph in $\Xi$ is precisely the number of vertices which have $d^+(v)=d^-(v)=2$.
We let $\Xi_k$ be the set $\Xi_k^h := \{ F \in \Xi | |F|=h\text{ and }\excess(F)=k \}$.
Moreover, for each $F \in \Xi$ define
\begin{equation}\label{eq:count_F_fromsamecycles}
t(F) := \prod_{v\in V(F)} d^-_G(v)^{d^-_F(v)} d^+_G(v)^{d^+_F(v)}
\end{equation}
Observe that for a given $C, C'$ whose union is $F$, $t(F)$ is the number of pairs of principal cycles $\widetilde{C}, \widetilde{C}'$ which are copies of the same cycles as $C, C'$.
Thus, combining~\eqref{eq:count_F_fromsamecycles} with~\Cref{lem:27k_cycle_union} allows us to bound $\eta$ as follows,
\begin{equation}\label{eq:eta_bound1}
 \eta = \sum_{C \in \Gamma} \sum_{C' \in \Gamma_C^s \cup \{ C \}} p_C p_{C'} \leq \sum_{h=\frac{\alpha}{|Q|}}^{2g(n)} \sum_{k=1}^h \sum_{F \in \Xi_k^h} \frac{27^k t(F)}{(m-h-k)^{h+k}} 
\end{equation}
Now, let $\Lambda_k^h$ be the set of all strongly connected labelled multi-digraphs with $h$ vertices such that all vertices have degrees $d^+(v)=d^-(v)=1$ or $d^+(v)=d^-(v)=2$ with precisely $k$ vertices $v$ such that $d^+(v)=d^-(v)=2$.
This allows us to write
\begin{equation}\label{eq:Lambda_Xi_relation}
 \sum_{F \in \Xi_k^h} t(F) = \frac{1}{h!}\sum_{F \in \Lambda_k^h} \sum_{\phi:V(F) \hookrightarrow V(G)} \prod_{v \in V(F)} d^-_G(v)^{d^-_F(v)} d^+_G(v)^{d^+_F(v)},
\end{equation}
where the division by $h!$ comes from the fact that $\Lambda_k^h$ distinguishes between digraphs obtained from one-another by a permutation of the vertex set.
Arguing similarly to~\Cref{lem:subgraphs_ub}, noting that we know the degree sequence of $F$ and using $\mu_{1,1} = \frac{m}{n} (1+Q) \leq \frac{m}{n}$, we deduce that
$$
\sum_{\phi:V(F) \hookrightarrow V(G)} \prod_{v \in V(F)} d^-_G(v)^{d^-_F(v)} d^+_G(v)^{d^+_F(v)}
\leq n^h \mu_{1,1}^{h-k} \mu_{2,2}^k \leq m^h \mu_{2,2}^k\mu_{1,1}^{-k}.
$$
Substituting into~\eqref{eq:Lambda_Xi_relation} and applying~\Cref{lem:count_scdegseq} we find
\begin{align}
 \sum_{F \in \Xi_k^h} t(F) 
 \leq  \frac{k}{h+k} \frac{(h+k)!}{h!} \binom{h}{k} m^h \mu_{2,2}^k \mu_{1,1}^{-k} 
 & = \frac{k}{h+k} \binom{h+k}{2k} \frac{(2k)!}{k!} m^h \mu_{2,2}^k \mu_{1,1}^{-k} \nonumber \\
 & \leq \frac{(h+k)^{2k-1}}{(k-1)!} m^h \mu_{2,2}^k \mu_{1,1}^{-k}. \label{eq:eta_bound2}
\end{align}

To finish, we note that $((m-h-k)^{h+k}) = (1+o(1))m^{h+k}$ as $h+k = o(m^{1/2})$, this allows us to substitute the bound found in~\eqref{eq:eta_bound2} into~\eqref{eq:eta_bound1} where we deduce
\begin{align}
\eta 
& \leq (1+o(1))\sum_{h=\frac{\alpha}{|Q|}}^{2g(n)} \sum_{k=1}^h \frac{27^k}{m^{h+k}} \frac{(h+k)^{2k-1}}{(k-1)!} m^h \mu_{2,2}^k \mu_{1,1}^{-k} \nonumber \\
& \leq (1+o(1)) \frac{216 g(n) \mu_{2,2}}{m} \sum_{h=\frac{\alpha}{|Q|}}^{2g(n)} \sum_{k=0}^h \frac{108^k h^{2k} \mu_{2,2}^k}{m^k k! \mu_{1,1}^{k}}. \label{eq:eta_bound3}
\end{align}
Part of the expression in~\eqref{eq:eta_bound3} is in the form of an exponential sum.
Evaluating this sum allows us to deduce that $\eta = o(1)$.
\begin{equation}\label{eq:eta_bound4}
 \eta 
 \leq (1+o(1))\frac{216 g(n) \mu_{2,2}}{m\mu_{1,1}} \sum_{h=\frac{\alpha}{|Q|}}^{2g(n)} e^{\frac{108h^2 \mu_{2,2}}{m\mu_{1,1}}}
 \leq (1+o(1)) \frac{432 g(n)^2 \mu_{2,2}}{m \mu_{1,1}} e^{\frac{432g(n)^2 \mu_{2,2}}{m\mu_{1,1}}} = o(1),
\end{equation}
where we deduce that this is $o(1)$ in the same way as when we bound~\eqref{eq:complex_expectation_deg4}.

Next we shall consider the term,
\begin{equation}\label{eq:theta_defn}
\theta := \sum_{C \in \Gamma} \bE(X_C Z_C) = \sum_{C \in \Gamma} \sum_{C' \in \Gamma_C^s \cup \{ C \}} \bE(X_C X_C')
\end{equation}
So note that $\bE(X_CX_{C'})$ is the probability that both $C$ and $C'$ are simultaneously present.
Furthermore, note that $C \cup C'$ is a strongly connected (not necessarily simple) digraph with maximum total degree at most $4$.
Thus we can use a similar strategy to the one used to bound $\eta$.
So define $\Theta$ to be the set of all strongly connected multi-digraphs $F$ with $V(F) \subseteq V(G)$ and $\Delta(F) \leq 4$.
Furthermore for $a,b \in\bN$, define 
$$
\Theta_{a,b}^h := \{ F \in \Theta | |F|=h, n_{1,2}(F)=n_{2,1}(F)=a, n_{2,2}(F)=b \},
$$
where we define $\Theta_{0,0}^h=\emptyset$.
Define $t(F)$ for $F \in \Theta$ in the same way as~\eqref{eq:count_F_fromsamecycles}.
Then, for a given construction, $F = C \cup C'$ from a pair of principal cycles, $t(F)$ is again the number of pairs of principal cycles $\widetilde{C}, \widetilde{C}'$ which are copies of the same cycles as $C, C'$.
Thus, if we apply~\Cref{lem:27k_cycle_union} we get the following bound on $\theta$,
\begin{equation}\label{eq:theta_bound1}
 \theta = \sum_{C \in \Gamma} \sum_{C' \in \Gamma_C^s} \bE(X_CX_{C'})
 \leq \sum_{h=\frac{\alpha}{|Q|}}^{2g(n)} \sum_{a=0}^h\sum_{b=0}^h \sum_{F \in \Theta_{a,b}^h} \frac{27^k t(F)}{(m-h-a-b)^{h+a+b}} 
\end{equation}
Now, we define $\Lambda_{a,b}^h$ to be the set of all strongly connected labelled multi-digraphs $F$ with $h$ vertices such that $n_{1,1}(F) = h-2a-b$, $n_{1,2}(F) = n_{2,1}(F) = a$ and $n_{2,2}(F)=b$.
(Note $\Lambda_{0,k}^h = \Lambda_k^h$).
Thus, arguing as previously,
\begin{equation}
 \sum_{F \in \Theta_{a,b}^h} t(F) = \frac{1}{h!} \sum_{F\in \Lambda_{a,b}^h} \sum_{\phi:V(F) \hookrightarrow V(G)} \prod_{v \in V(F)} d^-_G(v)^{d^-_F(v)} d^+_G(v)^{d^+_F(v)}
\end{equation}
Again, we may argue similarly to~\Cref{lem:subgraphs_ub} to deduce that for $F \in \Lambda_{a,b}^h$,
\begin{equation*}
 \sum_{\phi:V(F) \hookrightarrow V(G)} \prod_{v \in V(F)} d^-_G(v)^{d^-_F(v)} d^+_G(v)^{d^+_F(v)}
\leq m^h \mu_{1,2}^a \mu_{2,1}^a \mu_{2,2}^b \mu_{1,1}^{-(2a+b)}.
\end{equation*}
We can substitute this into~\eqref{eq:theta_bound1} and apply~\Cref{lem:count_scdegseq} to find
\begin{align}
 \sum_{F \in \Theta_{a,b}^h} t(F) & \leq \frac{3a+2b}{h+a+b} \frac{(h+a+b)!}{h!} \binom{h}{a,a,b} m^h \mu_{1,2}^a \mu_{2,1}^a \mu_{2,2}^b \mu_{1,1}^{-(2a+b)} \nonumber\\
 & \leq \frac{3a+2b}{h} 2^{a+b} \frac{h^{3a+2b}}{a!a!b!}m^h \mu_{1,2}^a \mu_{2,1}^a \mu_{2,2}^b \mu_{1,1}^{-(2a+b)}\nonumber\\ 
 & =  \frac{3a+2b}{h} \frac{1}{a!a!b!} \bbl\frac{2 h^3\mu_{1,2} \mu_{2,1}}{\mu_{1,1}^2}\bbr^a \bbl\frac{2 h^2\mu_{2,2}}{\mu_{1,1}}\bbr^b m^h\label{eq:tF_bound}
\end{align}
Note that this bound is a sum of two terms due to the factor $3a+2b$ in~\eqref{eq:tF_bound}.
So we will split this bound into $t_a(F)+t_b(F)$ in the obvious way.
This allows us to bound $\theta \leq \theta_a + \theta_b$ by only considering the $t_a(F)$ or $t_b(F)$ terms which will be convenient for us.
We substitute the bound from~\eqref{eq:tF_bound} into~\eqref{eq:theta_bound1} and use that $(m-h-a-b)^{h+a+b} = (1+o(1))m^{h+a+b}$ to deduce
\begin{align}
 \label{eq:theta_a} \theta_a & \leq (1+o(1)) \sum_{h=\frac{\alpha}{|Q|}}^{2g(n)} \sum_{a=0}^h\sum_{b=0}^h  \frac{3a}{h} \frac{1}{a!a!b!} \bbl\frac{2\cdot 27^2 h^3\mu_{1,2} \mu_{2,1}}{m \mu_{1,1}^2}\bbr^a \bbl\frac{54 h^2\mu_{2,2}}{m \mu_{1,1}}\bbr^b \\
 \label{eq:theta_b} \theta_b & \leq (1+o(1)) \sum_{h=\frac{\alpha}{|Q|}}^{2g(n)} \sum_{a=0}^h\sum_{b=0}^h  \frac{2b}{h} \frac{1}{a!a!b!} \bbl\frac{2\cdot 27^2 h^3\mu_{1,2} \mu_{2,1}}{m \mu_{1,1}^2}\bbr^a \bbl\frac{54 h^2\mu_{2,2}}{m \mu_{1,1}}\bbr^b
\end{align}
Note that in both cases we can evaluate the two inner sums in terms of exponential functions and modified Bessel functions of the first kind.
However the latter of these is a little difficult to work with, hence we will replace the $a!a!$ with $(2a)!$ allowing us to use hyperbolic trigonometric functions instead.
In particular, we have
$$
\frac{1}{a!a!} \leq \frac{4^a}{(2a)!}\;\text{ and }\; \frac{1}{a!(a-1)!} \leq \frac{4^a}{(2a-1)!}.
$$
Looking first at $\theta_a$ we get the bound
\begin{align}
 \theta_a & \leq (1+o(1)) \sum_{h=\frac{\alpha}{|Q|}}^{2g(n)} \sum_{a=0}^h\sum_{b=0}^h  \frac{3}{h} \frac{1}{(2a-1)!b!} \bbl\frac{8\cdot 27^2 h^3\mu_{1,2} \mu_{2,1}}{m \mu_{1,1}^2}\bbr^a \bbl\frac{54 h^2\mu_{2,2}}{m \mu_{1,1}}\bbr^b \nonumber \\
 & = (1+o(1))162\sqrt{2} \sum_{h=\frac{\alpha}{|Q|}}^{2g(n)} \sqrt{\frac{h \mu_{1,2} \mu_{2,1}}{m \mu_{1,1}^2}} \sinh\left(54 \sqrt{\frac{2h^3 \mu_{1,2} \mu_{2,1}}{m \mu_{1,1}^2}} \right) e^{\frac{54 h^2 \mu_{2,2}}{m \mu_{1,1}}} \nonumber \\
 & \leq (1+o(1))648 \sqrt{\frac{g(n)^3 \mu_{1,2} \mu_{2,1}}{m \mu_{1,1}^2}} \sinh\left(216 \sqrt{\frac{g(n)^3 \mu_{1,2} \mu_{2,1}}{m \mu_{1,1}^2}} \right) e^{\frac{216 g(n)^2 \mu_{2,2}}{m \mu_{1,1}}}, \label{eq:theta_a2}
\end{align}
where the final inequality follows from the fact that $\sqrt{x}$, $\sinh(x)$ and $e^x$ are all increasing for $x>0$.
Bounding $\theta_b$ is similar,
\begin{align}
 \theta_b & \leq (1+o(1)) \sum_{h=\frac{\alpha}{|Q|}}^{2g(n)} \sum_{a=0}^h\sum_{b=1}^h  \frac{2}{h} \frac{1}{(2a)!(b-1)!} \bbl\frac{8\cdot 27^2 h^3\mu_{1,2} \mu_{2,1}}{m \mu_{1,1}^2}\bbr^a \bbl\frac{54 h^2\mu_{2,2}}{m \mu_{1,1}}\bbr^b \nonumber \\
 & = (1+o(1)) 108 \sum_{h=\frac{\alpha}{|Q|}}^{2g(n)} \frac{h \mu_{2,2}}{m\mu_{1,1}} \cosh\left( 54\sqrt{\frac{2h^3 \mu_{1,2} \mu_{2,1}}{m \mu_{1,1}^2}} \right) e^{\frac{54h^2\mu_{2,2}}{m\mu_{1,1}}} \nonumber \\
 & \leq (1+o(1))432 \frac{g(n)^2 \mu_{2,2}}{m \mu_{1,1}} \cosh\left( 216\sqrt{\frac{g(n)^3 \mu_{1,2} \mu_{2,1}}{m \mu_{1,1}^2}} \right) e^{\frac{216g(n)^2\mu_{2,2}}{m\mu_{1,1}}} \label{eq:theta_b2}
\end{align}
Both~\eqref{eq:theta_a2} and~\eqref{eq:theta_b2} are $o(1)$ which follows from the facts that
$$
\frac{g(n)^2 \mu_{2,2}}{m \mu_{1,1}} \to 0 \qquad\text{ and }\qquad \frac{g(n)^3 \mu_{1,2} \mu_{2,1}}{m \mu_{1,1}^2} \to 0.
$$
The first of which we showed in~\eqref{eq:eta_bound4} and the latter follows directly from the assumption that $g(n)^3 = o(\frac{m}{R^-R^+})$.
Hence $\theta = o(1)$.
Finally, we will bound the terms,
\begin{equation}\label{eq:kappa_defn}
\kappa_C :=  \bE|\widetilde{W}_C^1 - W_C^1|.
\end{equation}
We note that unlike in bounding $\eta$ and $\theta$ we will not need to look at the sum over $C \in \Gamma$ in computing the bound.
This is due to the fact that $\kappa_C$ depends on a much more global structure than the terms $p_C \bE(X_C+Z_C)$ and $\bE(X_C Z_C)$.
For the bounding of $\kappa_C$, the first thing to observe is that due to the choice of coupling and $\Gamma_C^s$ we have $\widetilde{W}_C^1 \geq W_C^1$.
This is because all of the edges which are removed by the switchings are incident to a vertex of $C$ and therefore none of the principal cycles containing these edges are in $\Gamma_C^w$ which are the cycles contributing to $\widetilde{W}_C^1 \geq W_C^1$.
This enables us to bound $\kappa_C$ by computing the number of expected cycles from $\Gamma_C^w$ which are added by the switchings.

In order to add a cycle $C'$ with the switchings over $4$-cycles which produce $\widetilde{G}_C$ from $G'$ the following must be true for some $k$,
\begin{itemize}
 \item All but $k$ edges of $C'$ must be present in $G'$.
 \item $k$ of the edges of $C$ must add these edges after applying the switching.
\end{itemize}
Now let us compute the probability that this event comes to pass.
So let us fix $k$ edges missing from $C'$ which we match up with $k$ of the edges of $C$ which will be used to add them when we apply the switching.
So if one such edge is $e=uv$ and this is matched with $e'=u'v'$, in order for the switching to add $e$, before switching we must have edges $uv'$ and $u'v$ present in $G'$ (note that because the edges are directed we do not need to consider the  alternate switching using $uu'$ and $vv'$ like we would if working with graphs).
Hence in total there are $|C'|+k$ edges which we must find in $G'$ before switching in order that $C$ is a cycle of $\widetilde{G}_C$.
The probability we find these edges is $((m)_{|C'|+k})^{-1}$.

Now, let us count how many such structures there are which produce the same union cycle as $C'$.
As $C$ is a principal cycle, we know exactly which stubs we use for it and so there is no contribution to the number of copies from vertices of $C$.
However there are $d^+(u) d^-(v)$ switchings which add the union edge $uv$ and so the number of ways in which the switching structure with a union cycle which is that of $C'$ can be found is
$$
\prod_{v \in V(C')} d^-(v) d^+(v).
$$
Thus for any fixed set of missing edges for the union cycle with the same edges as $C'$ and choice of edges from $C$ with which we add these edges, the expected number of additional principal cycles after the switchings is
\begin{equation}
 \frac{1}{(m)_{|C'|+k}} \prod_{v \in V(C')} d^-(v) d^+(v).
\end{equation}
There are $\binom{|C|}{k}$ ways to pick the edges of $C$ which we switch over to add the missing edges of $C'$.
Also there are $\binom{|C'|}{k}$ ways to pick the missing edges of $C'$ and $k!$ ways to assign edges of $C$ to missing edges of $C'$.
Thus, when we sum over principal cycles $C'$, missing edges of $C$ and $C'$ and matchings of the missing edges, we find the expected number of such structures is at most
\begin{equation}
 \bE|\widetilde{W}_C^1 - W_C^1| \leq \sum_{h = \frac{\alpha}{|Q|}}^{g(n)} \sum_{k=1}^{\min(|C|,h)} \frac{1}{h}\sum_{C' \subseteq G; |C'|=h}\binom{|C|}{k} \binom{h}{k} \frac{k!}{(m)_{h+k}} \prod_{v \in V(C')} d^-(v) d^+(v). \label{eq:kap1}
 \end{equation}
 Applying standard bounds on binomial coefficients and falling factorials and taking vertices from $V(G(\bd))$ for $|C'|$ with replacement rather than without allows us to bound~\eqref{eq:kap1} as follows
 \begin{align}
 \bE|\widetilde{W}_C^1 - W_C^1| 
 & \leq (1+o(1)) \sum_{h = \frac{\alpha}{|Q|}}^{g(n)} \sum_{k=1}^{\infty} \frac{|C|^k h^{k-1}}{k! m^k} (1+Q)^h\nonumber \\
 & \leq (1+o(1)) \sum_{|C'| = \frac{\alpha}{|Q|}}^{g(n)} \frac{e^{\frac{|C|h}{m}}-1}{|C'|} \nonumber \\
 & \leq \frac{g(n)|Q|}{\alpha} \left(e^{\frac{g(n)^2}{m}} -1\right) \leq \frac{2g(n)^3 |Q|}{\alpha m} = o(1),\label{eq:kap2}
\end{align}
where we note that $g(n)^2/m < 1$ allows us to use the bound $e^x -1 \leq 2x$ in~\eqref{eq:kap2}.
So $\kappa_C \leq \kappa = o(1)$ for all $C$ where $\kappa = \frac{2g(n)^3 |Q|}{\alpha m}$.
Finally, note that this implies $\sum_{C \in \Gamma} p_C \kappa_C \leq \kappa\bE(W)$. 
We subsequently show that $\bE(W) \leq \xi_\alpha + o(1)$ from which it follows that $\sum_{C \in \Gamma} p_C \kappa_C = o(1)$ as $\xi_\alpha$ is a constant independent of $n$.
Using this and the fact that $\eta, \theta = o(1)$, we conclude that
$$
\dtv(W, Po(\bE(W))) = o(1).
$$
Note that $W$ is the number of cycles of $G(\bd)$ of length between $\alpha/|Q|$ and $g(n)$.
By~\Cref{lem:long_cyc} and~\Cref{lem:med_cyc}, the probability that there are any longer cycles is $o(1)$.
Thus, if $W'$ is the number of cycles of $G(\bd)$ of length at least $\alpha/|Q|$, $\dtv(W',Po(\bE(W))) = \dtv(W,Po(\bE(W))) + o(1) = o(1)$.
To show that the mean of the corresponding Poisson distribution can be taken to be $\xi_\alpha$ note that $\dtv(Po(\lambda),Po(\mu)) \leq |\lambda-\mu|$ holds for all $\lambda, \mu \geq 0$.
This follows from coupling $\Bin(n, \lambda/n)$ and $\Bin(n, \mu/n)$ by coupling their constituent Bernoulli trials and noting that $\dtv(\Bin(n,\lambda/n),Po(\lambda)) = o_n(1)$.
Thus it suffices to show that $\bE(W) = \xi_\alpha + o(1)$.
Upper bounding $\bE(W)$ may be done in the same way as the proof of~\Cref{lem:med_cyc}.
However, we can be more precise than we are in~\eqref{eq:cycle_expectation} as in this case $2h^2/m = o(1)$, so this line can be replaced by
\begin{equation}\label{eq:eW_ub}
 \sum_{h=\frac{\alpha}{|Q|}}^{g(n)}  \frac{e^{hQ+ \frac{2h^2}{m}}}{h}
 \leq \sum_{h=\frac{\alpha}{|Q|}}^{g(n)}  \frac{e^{hQ+o(1)}}{h}
 \leq (1+o(1)) \int_{\frac{\alpha}{|Q|}}^\infty \frac{e^{hQ}}{h} dh
 = (1+o(1)) \int_{\alpha}^\infty \frac{e^{-\lambda}}{\lambda} d\lambda = \xi_\alpha + o(1).
\end{equation}
Next, we lower bound $\bE(W)$ by $\xi_\alpha - o(1)$.
In order to do this we will use~\Cref{lem:cycles_lb} in combination with the inequality
\begin{equation*}
 1-x \geq e^{-x-x^2}\;\text{ for }\; x \leq \frac{1}{2}.
\end{equation*}
This allows us to deduce that the probability of finding a cycle of length between $\alpha/|Q|$ and $g(n)$ is at least
\begin{equation}\label{eq:eW_lb1}
 \sum_{h=\frac{\alpha}{|Q|}}^{g(n)} \frac{h!}{h} \binom{n}{h} \frac{(1+Q)^h}{(n)_h} \left( 1 - \frac{2h^2\Delta^2}{\eps n} \right) \left( 1 - \frac{h \Delta^2}{2m}\right) \geq \sum_{h=\frac{\alpha}{|Q|}}^{g(n)} \frac{(1+o(1))}{h} e^{h(Q - \frac{Q^2}{2})}
\end{equation}
Now, that the exponent in the above is equal to $hQ-o(1)$ follows from the assumption $g(n) = o(1/|Q|^2)$.
Thus in analogy with~\eqref{eq:eW_ub}, we deduce that~\eqref{eq:eW_lb1} is at most
\begin{equation}\label{eq:eW_lb2}
 (1+o(1))\sum_{h=\frac{\alpha}{|Q|}}^{g(n)}  \frac{e^{hQ-o(1)}}{h}
 \geq (1-o(1)) \int_{\frac{\alpha}{|Q|}}^\infty \frac{e^{hQ}}{h} dh
 = (1-o(1)) \int_{\alpha}^\infty \frac{e^{-\lambda}}{\lambda} d\lambda = \xi_\alpha - o(1).
\end{equation}
Combining~\eqref{eq:eW_ub} and~\eqref{eq:eW_lb2} we deduce that $\bE(W) = \xi_\alpha + o(1)$ as required.

\subsection{Putting it all together}

We can deduce the statement of the theorem from the previous sections as follows.
First we apply~\Cref{lem:long_cyc},~\Cref{lem:med_cyc} and~\Cref{lem:com_com} to deduce there are no cycles with length at least $\omega(1/|Q|)$ or complex components.
Then, we can apply~\Cref{lem:kth_cyc} to deduce that the distribution of the number of cycles with at least $\alpha/|Q|$ vertices converges to a Poisson distribution with mean $\xi_\alpha$ from which the statement that
$$
\bP\left(|\cC_k| \geq \frac{\alpha}{|Q|}\right) = 1-\sum_{i=0}^{k-1} \frac{\xi_\alpha^i}{i!} e^{-\xi_\alpha} +o(1)
$$
follows immediately as the above is simply the probability that a Poisson($\xi_\alpha$) distribution is at least $k$ (plus the $o(1)$ term).

\section{Concluding Remarks}\label{sec:concl}
In this paper, we showed that the largest component of the directed configuration model is of order $|Q|^{-1}$ when $nQ^3(R^-R^+)\to-\infty$ and found the distribution of the size of the $k$th largest component for any $k$.
In a subsequent paper~\cite{supercrit} we shall show that under similar conditions to those in this paper that for degree sequences such that $nQ^3(R^-R^+)^{-1} \to \infty$ the largest component is of order $nQ^2(R^-R^+)^{-1}$.
This quantity matches the $1/|Q|$ we find in this paper if $n|Q|^3(R^-R^+)^{-1} \to c$ for some constant $c$ and suggests that one may find a critical window phenomenon for degree sequences with such parameters.
Note that in particular this is precisely the critical window of $D(n,p)$ (looking at typical degree sequences from the model $D(n,p)$).
Moreover, the recent result of Donderwinkel and Xie certainly seems to indicate that the point $Q=0$ lies inside a critical window.

An interesting question for further study would be to ask about the joint distribution of the largest strongly connected components in the directed configuration model with parameters as in this paper.
The fact that we find
$$
\bP\bbl|\cC_k| \geq \frac{\alpha}{|Q|}\bbr = 1 - \sum_{i=0}^{k-1} \frac{\xi_\alpha^i}{i!} e^{-\xi_\alpha} + o(1)
$$
seems rather suggestive of an underlying Poisson process.
As such we make the following conjecture.
\begin{conjecture}
 Let $G$ be a directed configuration model with parameters as in~\Cref{thm:subcrit_main_thm} and suppose the sizes of its components in descending order are given by the random variables $Z_1 \geq Z_2 \geq \ldots$.
 Suppose further that $X_1 \leq X_2 \leq \ldots$ are the points of a Poisson process of rate $1$ and $Y_1 \geq Y_2 \geq \ldots$ are the unique positive solutions to
 $$
 X_i = \int_{Y_i}^\infty \frac{e^{-x}}{x} dx.
 $$
 Then we have that
 $$
 (|Q|Z_1,|Q|Z_2, \ldots) \to (Y_1, Y_2, \ldots)  \text{ as } n \to \infty.
 $$
\end{conjecture}
Note that the above is also an open question for $D(n,p)$ with $p = (1-\eps)/n$ and $\eps \to 0$, $\eps^3n \to \infty$.

\bibliographystyle{plain}
\bibliography{DS}
\end{document}